\documentclass{scrartcl}

\usepackage[normalem]{ulem}

\usepackage{authblk}

\usepackage{amssymb,amsmath,amsthm,soul,color}
\usepackage{mathrsfs}
\usepackage[utf8]{inputenc}
\usepackage[colorlinks=true,urlcolor=blue,citecolor=red,linkcolor=blue]{hyperref}
\usepackage[colorinlistoftodos]{todonotes}

\usepackage[left=2.5cm, right=2.5cm, top=2cm, bottom=2cm]{geometry}

\usepackage[noadjust]{cite}

\newtheorem{Th}{Theorem}[section]

\newtheorem{Lem}[Th]{Lemma}

\newtheorem{Rem}[Th]{Remark}

\makeatletter
    
    \newcommand{\Rmnum}[1]{\expandafter\@slowromancap\romannumeral #1@}

\newcommand{\cC}{{\mathcal C}}

\newcommand{\cF}{{\mathcal F}}

\newcommand{\cH}{{\mathcal H}}

\newcommand{\cJ}{{\mathcal J}}

\newcommand{\cN}{{\mathcal N}}

\newcommand{\cT}{{\mathcal T}}

\newcommand{\R}{\mathbb{R}}

\newcommand{\Z}{\mathbb{Z}}

\newcommand{\weakto}{\rightharpoonup}

\numberwithin{equation}{section}

\DeclareMathOperator{\essinf}{\mathrm{ess}\,\mathrm{inf}}

\begin{document}

\title{Systems of coupled Schr\"odinger equations with sign-changing nonlinearities via classical Nehari manifold approach}

\author{Bartosz Bieganowski\thanks{Email address: \texttt{bartoszb@mat.umk.pl}}}
\affil{Nicolaus Copernicus University, Faculty of Mathematics and Computer Science, ul. Chopina 12/18, 87-100 Toru\'n, Poland}

\maketitle

\begin{abstract} 
We propose existence and multiplicity results for the system of Schr\"{o}dinger equations with sign-changing nonlinearities in bounded domains or in the whole space $\R^N$. In the bounded domain we utilize the classical approach via the Nehari manifold, which is (under our assumptions) a differentiable manifold of class $\cC^1$ and the Fountain theorem by Bartsch. In the space $\R^N$ we additionally need to assume the $\mathbb{Z}^N$-periodicity of potentials and our proofs are based on the concentration-compactness lemma by Lions and the Lusternik-Schnirelmann values. 
\medskip

\noindent \textbf{Keywords:} ground state, variational methods, system of Schr\"odinger equations, Nehari manifold, periodic potential
   
\noindent \textbf{AMS Subject Classification:}  Primary: 35Q60; Secondary: 35J20, 35Q55, 58E05, 35J47
\end{abstract}

\section{Introduction}
\setcounter{section}{1}

We consider the following system of coupled Schr\"odinger equations
\begin{equation}\label{eq}
\left\{ \begin{array}{ll}
-\Delta u + V_1 (x) u = f_1 (u) - |u|^{q-2}u + \lambda(x)v & \quad \mbox{in} \ \Omega, \\
-\Delta v + V_2 (x) v = f_2 (v) - |v|^{q-2}v + \lambda(x)u & \quad \mbox{in} \ \Omega, \\
u = v = 0 & \quad \mbox{on} \ \partial \Omega,
\end{array} \right.
\end{equation}
where $\Omega \subset \R^N$ is a bounded domain or $\Omega = \R^N$. Solutions of \eqref{eq} describe standing waves of the following nonlinear time-dependent system
$$
\left\{ \begin{array}{ll}
i \frac{\partial \Psi}{\partial t} = -\Delta \Psi + V_1 (x) \Psi - f_1 (\Psi) + |\Psi|^{q-2}\Psi + \lambda(x)\Phi & \quad (t,x) \in \R \times \Omega, \\
i \frac{\partial \Phi}{\partial t} = -\Delta \Phi + V_2 (x) \Phi - f_2 (\Phi) + |\Phi|^{q-2}\Phi + \lambda(x)\Psi & \quad (t,x) \in \R \times \Omega.
\end{array} \right.
$$
The studying of the existence of standing waves for nonlinear
Schr\"{o}dinger equations arises in various branches of mathematical physics and nonlinear topics (see eg. \cite{Doerfler, GoodmanWinsteinJNS2001, Kuchment, Malomed, Pankov, NonlinearPhotonicCrystals} and references therein). Recently many papers have been devoted to the study of standing waves of the Schr\"odinger equation and of the system of Schr\"odinger equations (see eg. \cite{BieganowskiMederski, AmbrosettiColorado, BartschDingPeriodic, BenciGrisantiMeicheletti, BuffoniJeanStuart, ChabrowskiSzulkin2002, ChenZouCalPDE2013, CotiZelati, GuoMederski, IkomaTanaka, KryszSzulkin, LiSzulkin, Liu, MaiaJDE2006, MederskiTMNA2014, MederskiNLS2014, Rabinowitz:1992, WillemZou} and references therein).

Recently, J. Peng, S. Chen and X. Tang (\cite{PengChenTang}) studied semiclassical states of a similar system
$$
\left\{ \begin{array}{ll}
-\varepsilon^2 \Delta u + a (x) u = |u|^{p-2}u + \mu(x)v & \quad \mbox{in} \ \R^N, \\
-\varepsilon^2 \Delta v + b (x) v = |v|^{p-2}v + \mu(x)u & \quad \mbox{in} \ \R^N, \\
u,v \in H^1 (\R^N), &
\end{array} \right.
$$
where $a,b,\mu \in \cC(\R^N)$ and $\varepsilon > 0$ is sufficiently small. J. M. do \'O and J. C. de Albuquerque considered a similar system to \eqref{eq} in $\R$:
$$
\left\{ \begin{array}{ll}
(-\Delta)^{1/2} u + V_1 (x) u = f_1 (u) + \lambda(x)v & \quad \mbox{in} \ \R, \\
(-\Delta)^{1/2} v + V_2 (x) v = f_2 (v) + \lambda(x)u & \quad \mbox{in} \ \R,
\end{array} \right.
$$ 
but with the square root of the Laplacian $(-\Delta)^{1/2}$ and $f_i$ with exponential critical growth (see \cite{doO}). Similar systems were also considered in \cite{AmbrosettiCeramiRuiz, LiTang, Zhang}, see also references therein.

Our aim is to provide existence and multiplicity results using classical techniques in the presence of external, positive potentials and sign-changing nonlinearities. We show that under classical assumption (V2) on $\lambda$ and in the presence of nonlinearities like $g(x,u) = |u|^{p-2} u - |u|^{q-2}u$, where $2 < q < p < 2^*$ classical techniques can be applied. We assume that
\begin{enumerate}
\item[(F1)] for $i \in \{1,2\}$, $f_i \in \cC^1 (\R)$ is such that
$$
|f_i' (u)| \leq c (1 + |u|^{p-2}) \quad \mbox{for all} \ u \in \R,
$$
where $2 < q < p < 2^* := \frac{2N}{N-2}$; in particular the inequality
$$
|f_i (u)| \leq c (1 + |u|^{p-1}) \quad \mbox{for all} \ u \in \R,
$$
also holds for some $c > 0$;
\item[(F2)] $f_i(u) = o(u)$ as $u \to 0$;
\item[(F3)] $\frac{F_i (u)}{|u|^q} \to \infty$ as $|u| \to \infty$;
\item[(F4)] $\frac{f_i (u)}{|u|^{q-1}}$ is increasing on $(-\infty, 0)$ and on $(0,\infty)$;
\item[(F5)] $f_i(-u) = - f_i(u)$ for all $u \in \R$.
\end{enumerate}

Observe that in view of (F4) we can easily show that
\begin{equation}\label{AR}
0 \leq q F_i (u) \leq f_i (u)u
\end{equation}
for any $u \in \R$.

We impose the following conditions on potentials
\begin{enumerate}
\item[(V1)] for $i\in \{1,2\}$, $\essinf_{x \in \Omega} V_i (x) > 0$ and $V_i \in L^\infty (\Omega)$;
\item[(V2)] $\lambda (x) \geq 0$ is measurable and satisfies
$$
\lambda(x) \leq \delta \sqrt{V_1(x) V_2(x)}
$$
for some $0 < \delta < 1$.
\end{enumerate}

For $\Omega = \R^N$ we assume additionaly that
\begin{enumerate}
\item[(V3)] $V_1, V_2, \lambda \ \mbox{are } \mathbb{Z}^N\mbox{-periodic.}$
\end{enumerate}

Observe that functions $u \mapsto f_1 (u) - |u|^{q-2}u$, $u \mapsto f_2 (u) - |u|^{q-2}u$ don't need to satisfy the Ambrosetti-Rabinowitz-type condition on the whole real line $\R$, e.g. take $f_1(u) = f_2(u) = |u|^{p-2}u$. However such a condition is satisfied for sufficiently large $u$, see Lemma \ref{lemR}.

We provide the following results in the case of bounded $\Omega$.

\begin{Th}\label{Th:Main1}
Assume that (F1)-(F5) and (V1)-(V2) hold, and $\Omega \subset \R^N$ is a bounded domain. Then there exists a ground state solution $(u_0,v_0)$ of \eqref{eq}, i.e. a critical point $(u_0,v_0)$ of the energy functional $\cJ$ being minimizer on the Nehari manifold
$$
\cN = \{ (u,v) \in H^1_0(\Omega) \times H^1_0 (\Omega) \setminus \{(0,0)\} \ : \ \cJ'(u,v) (u,v) = 0 \},
$$
where $\cJ$ is given by \eqref{J}. Moreover $u_0, v_0 \geq 0$.
\end{Th}

\begin{Th}\label{Th:Main2}
Assume that (F1)-(F5) and (V1)-(V2) hold, and $\Omega \subset \R^N$ is a bounded domain. Then there is a sequence of solutions $(u_n, v_n)$ such that
$$
\cJ(u_n, v_n) \to \infty \quad \mbox{as } n \to \infty,
$$
where $\cJ$ is given by \eqref{J}.
\end{Th}

We also obtain the following existence and multiplicity results in the case $\Omega = \R^N$.

\begin{Th}\label{Th:Main3}
Assume that (F1)-(F4) and (V1)-(V3) hold, and $\Omega = \R^N$. Then there exists a ground state solution $(u_0,v_0)$ of \eqref{eq}, i.e. a critical point $(u_0,v_0)$ of the energy functional $\cJ$ being minimizer on the Nehari manifold
$$
\cN = \{ (u,v) \in H^1(\R^N) \times H^1 (\R^N) \setminus \{(0,0)\} \ : \ \cJ'(u,v) (u,v) = 0 \},
$$
where $\cJ$ is given by \eqref{J}. Moreover $u, v \in \cC(\R^N)$ and there are constants $C, \alpha > 0$ such that
$$
|u(x)| + |v(x)| \leq C e^{-\alpha |x|}.
$$
\end{Th}

\begin{Th}\label{Th:Main4}
Assume that (F1)-(F5) and (V1)-(V3) hold, and $\Omega =\R^N$. Then there are infinitely many pairs $(\pm u, \pm v)$ of solutions which are geometrically distinct.
\end{Th}

We recall that solutions $(u_1, v_1), (u_2, v_2) \in H^1 (\R^N) \times H^1(\R^N)$ are geometrically distinct if $\mathcal{O}(u_1, v_1) \cap \mathcal{O}(u_2, v_2) = \emptyset$, where 
$$
\mathcal{O}(u,v) := \{ (u(\cdot - z), v(\cdot -z)) \ : \ z \in \mathbb{Z}^N \}
$$ 
is the orbit of $(u,v) \in H^1 (\R^N) \times H^1(\R^N)$ under the action of $(\mathbb{Z}^N, +)$. Obviously, in view of (V3), if $(u,v)$ is a solution then the whole orbit $\mathcal{O}(u,v)$ consists of solutions.

\begin{Rem}
Very similar results to Theorems \ref{Th:Main3} and \ref{Th:Main4} can be obtained in the same way for the system of fractional equations
$$
\left\{ \begin{array}{ll}
(-\Delta)^{\alpha / 2} u + V_1 (x) u = f_1 (u) - |u|^{q-2}u + \lambda(x)v & \quad \mbox{in} \ \R^N, \\
(-\Delta)^{\alpha / 2} v + V_2 (x) v = f_2 (v) - |v|^{q-2}v + \lambda(x)u & \quad \mbox{in} \ \R^N, \\
(u,v) \in H^{\alpha / 2} (\R^N) \times H^{\alpha / 2} (\R^N). & 
\end{array} \right.
$$
\end{Rem}

Our paper is organized as follows. The second section contains used notations and some preliminary facts about nonlinearities, potentials and properties of the Nehari manifold. Sections \ref{sect:3}, \ref{sect:4}, \ref{sect:5} and \ref{sect:6} contain proofs of main results - Theorems \ref{Th:Main1}, \ref{Th:Main2}, \ref{Th:Main3} and \ref{Th:Main4} respectively.

\section{Notations and preliminary facts}\label{sect:2}

Let
$$
E := H^1_0 (\Omega) \times H^1_0 (\Omega)
$$
and
$$
\| (u,v) \|^2 := \| u \|_1^2 + \|v\|_2^2, \quad (u,v) \in E,
$$
where
$$
\|u\|_i^2 = \int_{\Omega} | \nabla u|^2 \, dx + \int_{\Omega} V_i(x) u^2 \, dx, \quad i \in \{1,2\}.
$$
Recall that in the case $\Omega = \R^N$ we have $H_0^1 (\Omega) = H^1 (\R^N)$.

The energy functional $\cJ : E \rightarrow \R$ is given by
\begin{align}\label{J}
\cJ(u,v) = \frac{1}{2} \left( \| (u,v) \|^2 - 2 \int_{\Omega} \lambda(x) uv \, dx \right) - \int_{\Omega} F_1 (u) + F_2 (v) \, dx + \frac{1}{q} \int_{\Omega} |u|^q + |v|^q \, dx.
\end{align}
It is classical to check that $\cJ \in \cC^1 (E)$ and critical points of $\cJ$ are weak solutions of \eqref{eq}.
Let
$$
\cN := \{ (u,v) \in E \setminus \{(0,0)\} \ : \ \cJ'(u,v) (u,v) = 0 \}.
$$

\begin{Lem}
For $s \neq 0$ there holds
\begin{equation}\label{1.1}
f_i'(s)s^2 - f_i(s)s > (q-2) f_i(s)s, \quad i \in \{1,2\}.
\end{equation}
\end{Lem}

\begin{proof}
Let $\varphi_i(s) := \frac{f_i(s)}{|s|^{q-1}}$ for $s > 0$ and in view of (F4) we have
$$
\frac{d \varphi_i (s)}{ds} > 0.
$$
Hence
$$
f_i'(s) s^{q-1} - (q-1) f_i (s) s^{q-2} > 0
$$
for $s > 0$. So
$$
f_i'(s) s^2 - f_i(s)s - (q-2) f_i(s) s > 0
$$
and the conclusion follows for $s > 0$. Suppose now that $s < 0$. Then $-s > 0$ and
$$
f_i'(-s) (-s)^2 - f_i(-s)(-s) - (q-2) f_i(-s) (-s) > 0
$$
which implies that
$$
f_i'(s) s^2 - f_i(s)s - (q-2) f_i(s)s > 0,
$$
since (F5) holds.
\end{proof}

\begin{Lem}
There holds
\begin{equation}\label{1.2}
\|(u,v)\|^2 - 2 \int_{\Omega} \lambda(x) uv \, dx \geq (1-\delta) \| (u,v) \|^2.
\end{equation}
\end{Lem}

\begin{proof}
For any $(u,v)\in E$ we have
$$
-2 \int_{\Omega} \lambda(x) uv \, dx  \geq -2\delta \int_{\Omega} \sqrt{V_1(x) V_2(x)} |uv| \, dx \geq -\delta \left( \int_{\Omega} V_1(x) u^2 \, dx + \int_{\Omega} V_2(x) v^2 \, dx \right).
$$
Hence
$$
\|(u,v)\|^2 - 2 \int_{\Omega} \lambda(x) uv \, dx \geq (1-\delta) \| (u,v) \|^2.
$$
\end{proof}

\begin{Lem}\label{lem:1.3}
$\cN \subset E$ is a $\cC^1$-manifold.
\end{Lem}

\begin{proof}
Define
$$
\xi(u,v) := \cJ'(u,v)(u,v) = \|(u,v)\|^2 - 2 \int_{ \Omega} \lambda(x) uv \, dx - \int_{ \Omega } f_1(u)u \, dx - \int_{ \Omega } f_2(v)v \, dx + |u|_q^q + |v|_q^q.
$$
Obviously $\cN = \xi^{-1} (\{ 0 \}) \setminus \{(0,0)\}$. For $(u,v) \in \cN$ one has, using \eqref{1.1} and \eqref{1.2}
\begin{align*}
\xi'(u,v)(u,v) &= 2 \| (u,v) \|^2 - 4 \int_{\Omega} \lambda(x) uv \, dx \\
&\quad - \int_{\Omega} f_1(u)u + f_1 ' (u) u^2 \, dx - \int_{\Omega} f_2(v)v + f_2 ' (v) v^2 \, dx + q |u|_q^q + q |v|_q^q \\
&= - \int_{\Omega} f_1 ' (u)u^2  - f_1(u)u\, dx - \int_{\Omega} f_2 ' (v)v^2 -f_2(v)v \, dx + (q-2) |u|_q^q + (q-2) |v|_q^q \\
&< - (q-2) \int_{\Omega} f_1(u)u \, dx - (q-2) \int_{\Omega} f_2(v)v  \, dx + (q-2) |u|_q^q + (q-2) |v|_q^q \\
&= (q-2) \left( |u|_q^q + |v|_q^q - \int_{\Omega} f_1(u)u + f_2(u)u \, dx \right) < 0.
\end{align*}
Therefore $0$ is a regular value of $\xi$ and $\xi^{-1}(\{0\}) \setminus \{(0,0)\} = \cN$ is a $\cC^1$-manifold.
\end{proof}

\begin{Lem}\label{lem:1.4}
For every $\varepsilon > 0$ there is $C_\varepsilon > 0$ such that
$$
|F_i (s)| + |f_i (s)s| \leq \varepsilon |s|^2 + C_\varepsilon |s|^p,
$$
where $i \in \{1,2\}$.
\end{Lem}

\begin{proof}
The inequality follows immediately from (F1), (F2) and (F5).
\end{proof}

\begin{Lem}\label{lem:1.5}
There holds
$$
\inf_{(u,v) \in \cN} \| (u,v) \| \geq \rho > 0.
$$
\end{Lem}

\begin{proof}
Suppose that $(u_n,v_n) \in \cN$ is such that
$$
\|(u_n,v_n)\| \to 0.
$$
Hence $\|u_n\|_1 \to 0$ and $\|v_n\|_2 \to 0$.
In view of \eqref{1.2}
\begin{align*}
(1 - \delta) \|(u_n,v_n)\|^2 &\leq \|(u_n,v_n)\|^2 - 2 \int_{\Omega} \lambda(x)u_n v_n \, dx = \int_{\Omega} f_1(u_n)u_n + f_2(v_n)v_n \, dx - |u_n|_q^q - |v_n|_q^q \\ &\leq \int_{\Omega} f_1(u_n)u_n + f_2(v_n)v_n \, dx.
\end{align*}
Thus
\begin{align*}
\|(u_n,v_n)\|^2 &\leq \frac{1}{1-\delta} \int_{\Omega} f_1(u_n)u_n + f_2(v_n)v_n \, dx \leq C \left( \varepsilon\|u_n\|_1^2 + C_\varepsilon \|u_n\|_1^p + \varepsilon\|v_n\|_2^2 + C_\varepsilon \|v_n\|_2^p \right) \\ &= C \left( \varepsilon \|(u_n,v_n)\|^2 + C_\varepsilon \|u_n\|_1^p + C_\varepsilon \|v\|_2^p \right).
\end{align*}
Choose $\varepsilon > 0$ such that $1 - \varepsilon C > 0$. Then
\begin{align*}
(1 - \varepsilon C) \leq C_\varepsilon \frac{\|u_n\|_1^p + \|v_n\|_2^p}{\|(u_n, v_n)\|^2} &= C_\varepsilon \left( \frac{\|u_n\|_1^p}{\|u_n\|_1^2 + \|v_n\|_2^2} + \frac{\|v_n\|_2^p}{\|u_n\|_1^2 + \|v_n\|_2^2} \right) \\ &\leq C_\varepsilon \left( \|u_n\|_1^{p-2} + \|v_n\|_2^{p-2} \right) \to 0
\end{align*}
- a contradiction.
\end{proof}

\begin{Lem}\label{lem:2.6}
Suppose that $(u_0,v_0) \in \cN$ is a critical point of $\cJ \Big|_\cN : \cN \rightarrow \R$. Then $\cJ'(u_0,v_0) = 0$. 
\end{Lem}

\begin{proof}
Let
$$
\xi(u,v) := \cJ'(u,v)(u,v).
$$
Since $(u_0,v_0) \in \cN$ is a critical point of $\cJ \Big|_\cN$ there exists a Lagrange multiplier $\mu \in \R$ such that
$$
\cJ'(u_0,v_0) - \mu \xi'(u_0,v_0) = 0.
$$
Thus
$$
0 = \cJ'(u_0,v_0)(u_0,v_0) = \mu \xi'(u_0,v_0)(u_0,v_0).
$$
Taking into account that $\xi'(u_0,v_0)(u_0,v_0) < 0$ (see the proof of Lemma \ref{lem:1.3}) we get $\mu = 0$ and 
$$
\cJ'(u_0,v_0) = \mu \xi'(u_0,v_0) = 0.
$$
\end{proof}

\begin{Lem}\label{lem:1.7}
For every $(u,v) \in E \setminus \{(0,0)\}$ there is a unique $t > 0$ such that
$$
(tu,tv) \in \cN
$$
and $\cJ(tu,tv) = \max_{s \geq 0} \cJ(su,sv)$.
\end{Lem}

\begin{proof}
Take any $(u,v) \in E \setminus \{(0,0)\}$ and consider the function
$$
\varphi(t) := \cJ(tu,tv)
$$
for $t \geq 0$. Obviously $\varphi(0) = 0$ and
$$
\varphi(t) = \frac{t^2}{2} \|(u,v)\|^2 - t^2 \int_{\Omega} \lambda(x)uv \, dx - \int_{\Omega} F_1(tu) + F_2(tv) \, dx + \frac{t^q}{q} \int_{\Omega} |u|^q + |v|^q \, dx.
$$
In view of (F3) we have $\varphi(t) \to -\infty$ as $t \to \infty$. Using Lemma \ref{lem:1.4} and \eqref{1.2} we gets
$$
\varphi(t) \geq  C t^2
$$
for sufficiently small $t > 0$. Hence there is a maximum point $t_{\max}$ of $t \mapsto \cJ (tu,tv)$ in the interval $(0,\infty)$. While $\varphi$ is of $\cC^1$-class for such $t_{\max}$ we have
$$
0 = \varphi'(t_{\max}) = \cJ'(t_{\max} u, t_{\max} v) (u,v).
$$
Hence $(t_{\max} u, t_{\max} v) \in \cN$. In order to show the uniquencess it is enough to show that for any $(u,v) \in \cN$ the point $t=1$ is the unique maximum of $\varphi$. For $(u,v)  \in \cN$ and $t > 0$ we compute
\begin{align}
\varphi'(t) &= t \|(u,v)\|^2 - 2t \int_{\Omega} \lambda(x) uv \, dx - \int_{\Omega} f_1 (tu)u + f_2(tv)v \, dx + t^{q-1} \int_{\Omega} |u|^q + |v|^q \, dx \nonumber \\
&= \int_{\Omega} f_1(u)tu - f_1(tu)u \, dx + \int_{\Omega} f_2(v)tv - f_2(tv)v \, dx + (t^{q-1} - t) \int_{\Omega} |u|^q + |v|^q \, dx\label{1.3}
\end{align}
For $t > 1$ we have $t^{q-1} - t > 0$ and in view of \eqref{1.2} we have
$$
0 < (1-\delta) \|(u,v)\|^2 \leq \int_{\Omega} f_1(u)u + f_2(v)v \, dx - |u|_q^q + |v|_q^q
$$
and therefore
\begin{equation}\label{ineq}
\int_{\Omega} |u|^q + |v|^q \, dx < \int_{\Omega} f_1(u)u + f_2(v)v \, dx
\end{equation}
Combining \eqref{1.3} with \eqref{ineq} under assumption that $t > 1$ we get
\begin{align*}
\varphi'(t) &< \int_{\Omega} f_1(u)tu - f_1(tu)u \, dx + \int_{\Omega} f_2(v)tv - f_2(tv)v \, dx + (t^{q-1} - t) \int_{\Omega} f_1(u)u + f_2(v)v \, dx \\
&= \int_{\Omega} t^{q-1} f_1 (u)u - f_1(tu)u \, dx + \int_{\Omega} t^{q-1} f_2 (v)v - f_2(tv)v \, dx < 0,
\end{align*}
since (F4) holds. Similarly $\varphi'(t) > 0$ for $t \in (0,1)$ and the proof is completed.
\end{proof}

Define the ground state energy level as
$$
c := \inf_{(u,v) \in \cN} \cJ(u,v).
$$

\begin{Lem}
There holds
$$
c > 0.
$$
\end{Lem}

\begin{proof}
Take $(u,v) \in \cN$ and taking \eqref{AR} and \eqref{1.2} into account, we see that
\begin{align*}
\cJ(u,v) &\geq \frac{1}{2} \left( \|(u,v)\|^2 - 2 \int_{\Omega} \lambda(x) uv \, dx \right) - \frac{1}{q} \int_{\Omega} f_1(u)u + f_2(v)v \, dx + \frac{1}{q} \int_{\Omega} |u|^q + |v|^q \, dx \\
&= \left(\frac{1}{2} - \frac{1}{q} \right) \left( \|(u,v)\|^2 - 2 \int_{\Omega} \lambda(x) uv \, dx \right) \geq \left(\frac{1}{2} - \frac{1}{q} \right) (1 - \delta) \|(u,v)\|^2.
\end{align*}
Hence the statement follows by Lemma \ref{lem:1.5}.
\end{proof}

\begin{Rem}
Observe that from the inequality
$$
\cJ(u,v) \geq \left(\frac{1}{2} - \frac{1}{q} \right) (1 - \delta) \|(u,v)\|^2
$$
it follows that $\cJ$ is coercive, i.e. $\{ (u_n,v_n) \}_{n \geq 1} \subset \cN$ and $\|(u_n, v_n)\| \to \infty$ imply that $$\cJ(u_n, v_n) \to \infty.$$
\end{Rem}

\begin{Rem}\label{rem:1.10}
In view of the coercivity of $\cJ$ on $\cN$, any sequence $\{ (u_n, v_n) \}_{n \geq 1} \subset \cN$ such that $\cJ(u_n, v_n) \to c$ is bounded in $E$.
\end{Rem}

\section{Existence of a ground state in a bounded domain}\label{sect:3}

By Ekeland's variational principle there is a Palais-Smale sequence on $\cN$, i.e. a sequence $\{(u_n,v_n)\}_{n \geq 1}\subset\cN$~such that $\cJ(u_n, v_n) \to c$ and $\left( \cJ \Big|_\cN \right)' (u_n, v_n) \to 0$. Taking $(|u_n|, |v_n|)$ instead of $(u_n, v_n)$ we may assume that $u_n \geq 0$ and $v_n \geq 0$. In view of Remark \ref{rem:1.10} the sequence $\{ (u_n, v_n) \}_{n \geq 1} \subset \cN$ is bounded in $E$. 
Arguing as in Lemma \ref{lem:2.6} we see that $\left( \cJ \Big|_\cN \right)' (u_n, v_n) \to 0$ implies also $\cJ' (u_n, v_n) \to 0$. Hence $\{(u_n, v_n)\}_{n\geq 1}$ is a bounded Palais-Smale sequence for the free functional $\cJ$. Moreover $\cJ$ satisfies the Palais-Smale condition (see eg. \cite[Lemma 2.17]{Willem}) and $\{ (u_n, v_n) \}_{n \geq 1}$ has a convergent subsequence, i.e. (up to a subsequence)
$$
(u_n, v_n) \to (u_0,v_0) \quad \mbox{in} \ E.
$$
Hence $\cJ(u_n,v_n) \to \cJ(u_0,v_0)$ and therefore $\cJ(u_0,v_0) = c$. Thus $(u_0,v_0)$ is a ground state solution and obviously $u_0,v_0 \geq 0$.

\section{Multiplicity result in a bounded domain}\label{sect:4}

We will use the following Fountain Theorem provided by T. Bartsch. 

\begin{Th}[\cite{Bartsch}, {\cite[Theorem 3.6]{Willem}}]\label{Th:Bartsch}
Suppose that $X$ is a Banach space, $\cJ \in \cC^1 (X)$ and $G$ is a compact group. Moreover, assume that for any $k \in \mathbb{N}$ there are $\rho_k > r_k > 0$ such that
\begin{enumerate}
\item[(B1)] $G$ acts isometrically on
$$
X = \bigoplus_{j =0}^\infty X_j,
$$
where $X_j$ are $G$-invariant, $X_j$ are isomorphic to a finite dimensional space $V$ such that the action of $G$ on $V$ is admissible;
\item[(B2)] $a_k := \max_{u \in Y_k, \ \|u\|_X = \rho_k} \cJ(u) \leq 0$, where $Y_k = \bigoplus_{j=0}^k X_j$;
\item[(B3)] $b_k := \inf_{u \in Z_k, \ \|u\|_X = r_k} \cJ(u) \to \infty$ as $k \to \infty$, where $Z_k := \overline{\bigoplus_{j=k}^\infty X_j}$;
\item[(B4)] $\cJ$ satisfies the Palais-Smale condition at every level $c > 0$. 
\end{enumerate}
Then there exists an unbounded sequence of critical points of $\cJ$.
\end{Th}

\begin{Lem}\label{lemR}
There is a radius $R > 0$ such that
$$
0 < q \left( F_i(u) - \frac{1}{q} |u|^q \right) \leq f_i(u)u - |u|^q
$$
for $|u| \geq R$.
\end{Lem}

\begin{proof}
In view of (F3) we have $F_i (u) > \frac{1}{q} |u|^q$ for sufficiently large $|u| \geq R$. Hence the inequality follows by \eqref{AR}.
\end{proof}

Let $(e_j)$ be an orthonormal basis of $E = H^1_0 (\Omega) \times H^1_0 (\Omega)$, $G = \mathbb{Z}_2 := \mathbb{Z}/2\mathbb{Z}$ and $X_j := \mathbb{R}e_j$. On $E$ we consider the antipodal action of $G$. In view of the Borsuk-Ulam theorem the condition (B1) is satisfied. From Lemma \ref{lemR}, (F3) and \eqref{AR} there is $C > 0$ such that
$$
C ( |u|^q - 1) \leq F_i(u) - \frac{1}{q} |u|^q.
$$
Hence
\begin{align*}
\cJ (u,v) &\leq \frac{1}{2} \left( \|(u,v)\|^2 - 2 \int_\Omega \lambda(x) uv dx \right) - C (|u|_q^q+|v|_q^q) + 2C | \Omega| \\ 
&\leq \frac{1}{2} \left( \|(u,v)\|^2 + 2 \int_\Omega \lambda(x) |u| |v| dx \right) - C (|u|_q^q+|v|_q^q) + 2C | \Omega| \\
&\leq \frac{1}{2} \left( \|(u,v)\|^2 + |\lambda|_\infty \left( \int_\Omega |u|^2 \, dx + \int_\Omega |v|^2 \, dx \right) \right) - C (|u|_q^q+|v|_q^q) + 2C | \Omega|. 
\end{align*}
Since on finite dimensional space $Y_k$ all norms are equivalent, we get
$$
\cJ(u,v) \leq C_1 \|(u,v)\|^2 - C_2 \| (u,v)\|^q + C_3 \quad \mathrm{for} \ (u,v) \in Y_k.
$$
Hence the condition (B2) is satisfied for $\rho_k > 0$ large enough.
From (F1) there is $\tilde{C} > 0$ such that
$$
|F_i(u)| \leq \tilde{C} (1 + |u|^p).
$$
Put $\beta_k := \sup_{(u,v) \in Z_k, \ \|(u,v)\|=1} |u|_p + |v|_p$. Then
\begin{align*}
\cJ(u,v) &\geq \frac{1}{2} \left( \|(u,v)\|^2 - 2 \int_\Omega \lambda(x) uv \right) - \tilde{C} |u|_p^p - \tilde{C} |v|_p^p - 2 \tilde{C} |\Omega| \\
&\geq \frac{1-\delta}{2} \|(u,v)\|^2 - 2\tilde{C} \beta_k^p \|(u,v)\|^p - 2 \tilde{C} |\Omega|.
\end{align*}
Let $r_k := (2 \tilde{C} \frac{p}{1-\delta} \beta_k^p )^{1/(2-p)}$. Hence for $(u,v) \in Z_k$ and $\|(u,v)\| = r_k$ we get
\begin{align*}
\cJ(u,v) &\geq \frac{1-\delta}{2} \left(2 \tilde{C} \frac{p}{1-\delta} \beta_k^p \right)^{2/(2-p)} - 2 \tilde{C} \beta_k^p \left(2 \tilde{C} \frac{p}{1-\delta} \beta_k^p \right)^{p/(2-p)} - 2 \tilde{C} |\Omega| \\
&= \left( \frac{1-\delta}{2} -\frac{1-\delta}{ p } \right) \left(2 \tilde{C} \frac{p}{1-\delta} \beta_k^p \right)^{2/(2-p)}  - 2 \tilde{C} |\Omega| \\
&= (1-\delta) \left( \frac{1}{2} -\frac{1}{ p } \right) \left(2 \tilde{C} \frac{p}{1-\delta} \beta_k^p \right)^{2/(2-p)}  - 2 \tilde{C} |\Omega|.
\end{align*}
Hence it is enough to show that $\beta_k \to 0^+$. Clearly $0 \leq \beta_{k+1} \leq \beta_k$. Hence $\beta_k \to \beta$ and for any $k \geq 0$ there is $(u_k,v_k) \in Z_k$ such that $\|(u_k,v_k)\| = 1$ and $|u_k|_p + |v_k|_p > \frac{\beta_k}{2}$. In view of the definition of $Z_k$ we have $(u_k,v_k) \weakto (0,0)$ in $H^1_0 (\Omega) \times H^1_0 (\Omega)$. In view of Sobolev embeddings we obtain $|u_k|_p + |v_k|_p \to 0$ and therefore $\beta_k \to 0$ and (B3) is proved. It is classical to check that (B4) is satisfied, see e.g. \cite[Lemma 2.17]{Willem}.


Hence, in view of Theorem \ref{Th:Bartsch} and coercivity of $\cJ$ on $\cN$ there exists a sequence of solutions $(u_n, v_n)$ such that $\cJ(u_n,v_n) \to \infty$ and the proof of Theorem \ref{Th:Main2} is completed.

\section{Existence of a ground state in $\R^N$}\label{sect:5}

By Ekeland's variational principle there is a Palais-Smale sequence on $\cN$, i.e. a sequence $\{ (u_n, v_n) \}_{n \geq 1} \subset \cN$ such that $\cJ(u_n, v_n) \to c$ and $\left( \cJ \Big|_\cN \right)' (u_n, v_n) \to 0$. In view of Remark \ref{rem:1.10} the sequence $\{ (u_n, v_n) \}_{n \geq 1} \subset \cN$ is bounded in $E$. Passing to a subsequence we may assume that
\begin{align*}
(u_n, v_n) &\weakto (u_0, v_0) \quad \mbox{in} \ E, \\
(u_n, v_n) &\to (u_0, v_0) \quad \mbox{in} \ L^t_{loc} (\R^N) \times L^t_{loc} (\R^N) \ \mbox{for every} \ 2 \leq t < 2^*, \\
(u_n(x), v_n(x)) &\to (u_0(x), v_0(x)) \quad \mbox{for a.e. } x \in \Omega.
\end{align*}
Take any $(\varphi, \psi) \in \cC_0^\infty (\R^N) \times \cC_0^\infty (\R^N)$ and see that
\begin{align*}
\cJ' (u_n, v_n) ( \varphi, \psi) &= \langle (u_n, v_n), (\varphi, \psi) \rangle - \int_{\R^N} \lambda(x) u_n \psi \, dx - \int_{\R^N} \lambda(x) v_n \varphi \, dx \\ &\quad - \int_{\R^N} f_1 (u_n) \varphi + f_2 (v_n) \psi \, dx + \int_{\R^N} |u_n|^{q-2}u_n\varphi + |v_n|^{q-2}v_n\psi \, dx.
\end{align*}
In view of the weak convergence we have
\begin{align*}
\langle (u_n, v_n), (\varphi, \psi) \rangle &\to \langle (u_0, v_0), (\varphi, \psi) \rangle, \\
\int_{\R^N} \lambda(x) u_n \psi \, dx &\to \int_{\R^N} \lambda(x) u_0 \psi \, dx, \\
\int_{\R^N} \lambda(x) v_n \varphi \, dx &\to \int_{\R^N} \lambda(x) v_0 \varphi \, dx.
\end{align*}
From the Lebesgue's dominated convergence theorem there hold
\begin{align*}
\int_{\R^N} |u_n|^{q-2}u_n\varphi \, dx &\to \int_{\R^N} |u_0|^{q-2}u_0\varphi \, dx \\
\int_{\R^N} |v_n|^{q-2}v_n\psi \, dx &\to \int_{\R^N} |v_0|^{q-2}v_0\psi \, dx.
\end{align*}
Let $K \subset \R^N$ be a compact set containing supports of $\varphi$ and $\psi$. Then
$$
(u_n, v_n) \to (u_0, v_0) \quad \mbox{in} \ L^t (K) \times L^t (K) \ \mbox{for every} \ 2 \leq t < 2^*.
$$
From the continuity of the Nemytskii operator we obtain the convergence
$$
\int_{K} f_1 (u_n) \varphi  \, dx \to \int_{K} f_1 (u_0) \varphi  \, dx.
$$
Similarly 
$$
\int_{K} f_2 (v_n) \psi  \, dx \to \int_{K} f_2 (v_0) \psi  \, dx.
$$
Hence
$$
\cJ'(u_n,v_n) (\varphi, \psi) \to \cJ'(u_0, v_0) (\varphi, \psi).
$$


Similarly we can show that 
$$
\left( \cJ \Big|_\cN \right)' (u_n, v_n) \to \left( \cJ \Big|_\cN \right)' (u_0, v_0) 
$$
and therefore $\left( \cJ \Big|_\cN \right)' (u_0, v_0) = 0$. In view of Lemma \ref{lem:2.6} we obtain that $\cJ'(u_0,v_0) = 0$, i.e. $(u_0,v_0)$ is a critical point of $\cJ$. If $(u_0,v_0) \neq (0,0)$ we are done. Hence assume that $(u_0, v_0) = (0,0)$. We will use the following concentration-compactness result due to P.-L. Lions.

\begin{Lem}[{\cite[Lemma 1.21]{Willem}}]\label{lem:lions}
Let $r > 0$ and $2 \leq s < 2^*$. If $\{ w_n \}$ is bounded in $H^1 (\R^N)$ and if
\begin{equation}\label{lions}
\sup_{y \in \R^N} \int_{B(y,r)} |w_n|^s \, dx \to 0 \quad \mbox{as} \ n\to\infty,
\end{equation}
then $w_n \to 0$ in $L^t (\R^N)$ for $2 < t < 2^*$.
\end{Lem} 

Assume that
\begin{equation}\label{eq:5.2}
\sup_{y \in \R^N} \int_{B(y,1)} |u_n|^2 + |v_n|^2 \, dx \to 0 \quad \mbox{as} \ n \to \infty. 
\end{equation}
In view of Lemma \ref{lem:lions} we get $u_n \to 0$ and $v_n \to 0$ in $L^{t} (\R^N)$ for all $t \in (2,2^*)$. Then
\begin{align*}
(1-\delta) \| (u,v)\|^2 &\leq \| (u_n, v_n) \|^2 - 2 \int_{\R^N} \lambda(x) u_n v_n \, dx \\ &=  \int_{\R^N} f_1 (u_n)u_n + f_2 (v_n)v_n \, dx - \int_{\R^N} |u_n|^q + |v_n|^q \, dx \\
&= \int_{\R^N} f_1 (u_n)u_n + f_2 (v_n)v_n \, dx + o(1).
\end{align*}
From Lemma \ref{lem:1.4} we get
$$
\left| \int_{\R^N} f_1 (u_n) u_n \, dx \right| \leq \varepsilon |u_n|_2^2 + C_\varepsilon |u_n|_p^p.
$$
In view of boundedness of $\{ u_n \}$ we obtain that 
$$
\int_{\R^N} f_1 (u_n) u_n \, dx \to 0.
$$
Similarly 
$$
\int_{\R^N} f_2 (v_n) v_n \, dx \to 0
$$
and therefore $\|(u_n, v_n)\| \to 0$ - a contradiction with Lemma \ref{lem:1.5}. Hence \eqref{eq:5.2} cannot hold. Hence there is a sequence $(z_n) \subset \mathbb{Z}^N$ such that
\begin{equation}\label{eq:5.3}
\liminf_{n\to\infty} \int_{B(z_n, 1+\sqrt{N})} |u_n|^2 + |v_n|^2 \, dx > 0.
\end{equation}
It is classical to check that that $|z_n| \to \infty$. Moreover $(u_n (\cdot - z_n), v_n (\cdot - z_n) ) \weakto (\tilde{u}, \tilde{v})$ in $H$ and in view of \eqref{eq:5.3} we have $(\tilde{u}, \tilde{v}) \neq (0,0)$. Define $\tilde{u}_n := u_n (\cdot - z_n)$ and $\tilde{v}_n := v_n (\cdot - z_n)$. Then similarly as before
\begin{align*}
\cJ'(\tilde{u}_n, \tilde{v}_n) (\varphi, \psi) &\to \cJ'(\tilde{u}, \tilde{v}) (\varphi, \psi) \quad \mbox{for all } (\varphi, \psi) \in \cC_0^\infty (\R^N) \times \cC_0^\infty (\R^N).
\end{align*}
In view of $\mathbb{Z}^N$-periodicity of $V_1, V_2$ and $\lambda$ we also have 
\begin{align*}
\cJ'(\tilde{u}_n, \tilde{v}_n) (\varphi, \psi) &\to 0 \quad \mbox{for all } (\varphi, \psi) \in \cC_0^\infty (\R^N) \times \cC_0^\infty (\R^N).
\end{align*}
and therefore, in view of Lemma \ref{lem:2.6}, we obtain that $(\tilde{u}, \tilde{v})$ is a nontrivial critical point of $\cJ$, in particular $\cJ(\tilde{u},\tilde{v}) \geq c$. In view of $\mathbb{Z}^N$-periodicity of $V_1, V_2$ and $\lambda$ we have $\cJ(u_n, v_n) = \cJ(\tilde{u}_n, \tilde{v}_n) \to c$. If
$$
\sup_{y \in \R^N} \int_{B(y,1)} |\tilde{u}_n - \tilde{u}|^2 +  |\tilde{v}_n - \tilde{v}|^2 \, dx \to 0
$$
then in view of Lemma \ref{lem:lions} we obtain $\tilde{u}_n \to \tilde{u}$ and $\tilde{v}_n \to \tilde{v}$ in $L^t (\R^N)$ for all $t \in (2,2^*)$ and, as before, $(\tilde{u}_n, \tilde{v}_n) \to (\tilde{u}, \tilde{v})$ and $(\tilde{u}, \tilde{v})$ is a ground state. Otherwise there are $(\tilde{z}_n) \subset \Z^N$ such that
$$
\liminf_{n\to\infty} \int_{B(\tilde{z}_n, 1+\sqrt{N})} |\tilde{u}_n - \tilde{u}|^2 +  |\tilde{v}_n - \tilde{v}|^2 \, dx > 0
$$
and similarly
$$
(\bar{u}_n, \bar{v}_n) := (\tilde{u}_n (\cdot - \tilde{z}_n), \tilde{v}_n (\cdot - \tilde{z}_n)) \weakto (\bar{u},\bar{v}) \neq (0,0); \quad \cJ'(\bar{u},\bar{v}) = 0.
$$
Repeating this argument we obtain the following decomposition lemma (for more details see eg. \cite[Theorem 4.1]{BieganowskiMederski}).

\begin{Lem}\label{decomp}
There are $\ell \geq 0$, $(z_n^k) \subset \Z^N$ and $(w_1^k, w_2^k) \in E$, where $k = 1, \ldots, \ell$, such that
\begin{enumerate}
\item[(i)] $(w_1^k, w_2^k) \neq (0,0)$ and $\cJ' (w_1^k, w_2^k) = 0$;
\item[(ii)] $\left\| \left( u_n - u_0 - \sum_{k=1}^\ell w_1^k (\cdot - z_n^k), v_n - v_0 - \sum_{k=1}^\ell w_2^k (\cdot - z_n^k) \right) \right\| \to 0$;
\item[(iii)] $\cJ (u_n, v_n) \to \cJ(u_0,v_0) + \sum_{k=1}^\ell \cJ (w_1^k, w_2^k)$.
\end{enumerate}
\end{Lem}

While we assumed that $(u_0,v_0) = (0,0)$, from Lemma \ref{decomp}(iii) we get
$$
c + o(1) = \cJ(u_n, v_n) \to \sum_{k=1}^\ell \cJ (w_1^k, w_2^k) \geq \ell c.
$$
Hence $c \geq \ell c$ and therefore $\ell \in \{0,1\}$. While $(u_0,v_0) = 0$ we cannot have $\ell = 0$ and therefore $\ell = 1$, and $(w_1^1, w_2^1)$ is a ground state solution.


\cite[Theorem 2]{PankovDecay} gives the continuity and exponential decay of the solution.

\section{Multiplicity of solutions in $\R^N$}\label{sect:6}

To show Theorem \ref{Th:Main4} we will adapt the argument from \cite{SzulkinWeth} to our context. Let $\tau_k$ denote the action of $(\mathbb{Z}^N, +)$ on $E$, i.e.
$\tau_k (u,v) := (u(\cdot -k), v(\cdot - k))$, where $k \in \mathbb{Z}^N$. It is easy to show that
$$
\tau_k \cN \subset \cN,
$$ 
i.e. $\cN$ is invariant under $\tau_k$. Similarly $\| \tau_k (u,v) \| = \| (u,v)\|$ and $\cJ(\tau_k (u,v)) = \cJ(u,v)$. Since $\cJ$ is invariant, we know that $\nabla \cJ$ is equivariant.

Fix any $(u,v) \in E \setminus\{(0,0)\}$. Let
$$
S := \{ (u,v) \in E \ : \ \| (u,v)\| = 1 \}.
$$
Then there exists unique $t_{(u,v)} > 0$ such that $(t_{(u,v)}u,t_{(u,v)}v) \in \cN$. Define
$$
m : S \rightarrow \cN
$$
by the formula $m(u,v) := (t_{(u,v)} u,t_{(u,v)} v)$. Obviously $m$ is bijection and the inverse is given by
$$
m^{-1} (u, v) = \left( \frac{u}{\|(u,v)\|}, \frac{v}{\|(u,v)\|} \right).
$$

\begin{Lem}
The function $m : S \rightarrow \cN$ is a local diffeomorphism of class $\cC^1$.
\end{Lem}

\begin{proof}
Let $\xi : E \setminus \{(0,0)\} \rightarrow \R$ be given by
$$
\xi(u,v) := \cJ'(u,v)(u,v).
$$
Fix $(u,v) \in E \setminus \{(0,0)\}$. From the proof of Lemma \ref{lem:1.7} there is a unique $t_{(u,v)}$ such that
$$
\xi\left( t_{(u,v)} u, t_{(u,v)} v \right) = 0.
$$
From the Implicit Function Theorem
$$
E \ni (u,v) \mapsto t_{(u,v)} \in \R \setminus \{0\}
$$
is of $\cC^1$-class and therefore
$$
\hat{m} : E \setminus \{ (0,0) \} \rightarrow \cN, \quad \hat{m}(u,v) = \left( t_{(u,v)} u, t_{(u,v)} v \right)
$$
is of $\cC^1$-class. Clearly, the restriction $m = \hat{m} \Big|_{S}$ is a local diffeomorphism.
\end{proof}

Similarly as in \cite[Lemma 5.6]{Bieganowski} we show that $m : S \rightarrow \cN$, $m^{-1} : \cN \rightarrow S$ and $\nabla (\cJ \circ m) :S \rightarrow E$ are $\tau_k$-equivariant.

\begin{Lem}\label{lem:lip}
The function $m^{-1} : \cN \rightarrow S$ is Lipschitz continuous.
\end{Lem}

\begin{proof}
Fix $(u,v), (\tilde{u}, \tilde{v}) \in \cN$. See that
\begin{align*}
&\quad \left\| m^{-1}(u,v) - m^{-1} (\tilde{u},\tilde{v}) \right\| = \left\| \left( \frac{u}{\|(u,v)\|} - \frac{\tilde{u}}{\|(\tilde{u},\tilde{v})\|}, \frac{v}{\|(u,v)\|} - \frac{\tilde{v}}{\|(\tilde{u},\tilde{v})\|} \right) \right\| \\
&= \left\| \left( \frac{u-\tilde{u}}{\|(u,v)\|} + \frac{\tilde{u} \| (\tilde{u},\tilde{v})\| - \tilde{u} \| (u,v)\|}{\|(u,v)\| \cdot \|(\tilde{u},\tilde{v})\|}, \frac{v-\tilde{v}}{\|(u,v)\|} + \frac{\tilde{v} \| (\tilde{u},\tilde{v})\| - \tilde{v} \| (u,v)\|}{\|(u,v)\| \cdot \|(\tilde{u},\tilde{v})\|} \right) \right\| \\
&\leq \frac{\|(u-\tilde{u}, v - \tilde{v})\|}{\|(u,v)\|} + \frac{\left| \|(\tilde{u},\tilde{v})\| - \|(u,v)\| \right|}{\|(u,v)\|} \leq 2 \frac{\|(u-\tilde{u}, v - \tilde{v})\|}{\|(u,v)\|} \\
&\leq L \| (u - \tilde{u}, v - \tilde{v}) \|,
\end{align*}
where $L := \frac{2}{\rho} > 0$ and $\rho > 0$ is given by Lemma \ref{lem:1.5}.
\end{proof}

Let 
$$
\mathscr{C} := \{ (u,v) \in S \ : \ (\cJ \circ m)'(u) = 0 \}.
$$
Let $\cF \subset \mathscr{C}$ be a symmetric set such that for every orbit $\mathcal{O}(u,v)$ there is unique representative $v \in \cF$. We want to show that $\cF$ is infinite. Assume by contradiction that $\cF$ is finite. Then we have that (see \cite{SzulkinWeth})
$$
\kappa := \inf \{ \|(u-\tilde{u}, v - \tilde{v}) \| \ : \ (u,v),(\tilde{u},\tilde{v}) \in \mathscr{C}, \ (u,v) \neq (\tilde{u}, \tilde{v}) \} > 0.
$$
Hence $\mathscr{C}$ is a discrete set.

\begin{Lem}
Let $d \geq c = \inf_{\cN} \cJ$. If $(w_n^1, z_n^1), (w_n^2, z_n^2) \subset S$ are Palais-Smale sequences for $\cJ \circ m$ such that
$$
(\cJ \circ m)(w_n^i, z_n^i) \leq d, \quad i \in \{1,2\},
$$
then
$$
\|(w_n^1 - w_n^2, z_n^1 - z_n^2)\| \to 0
$$
or
$$
\liminf_{n\to\infty} \|(w_n^1 - w_n^2, z_n^1 - z_n^2)\| \geq \rho(d) > 0,
$$
where the constant $\rho(d) > 0$ depends only on $d$, but not on the particular choice of sequences.
\end{Lem}

\begin{proof}
Define $(u_n^i, v_n^i) := m ( w_n^i, z_n^i )$ for $i \in \{1,2\}$. Then a similar reasoning to \cite[Corollary 2.10]{SzulkinWeth} shows that $(u_n^i, v_n^i)$ are Palais-Smale sequences for $\cJ$ and 
$$
\cJ (u_n^i, v_n^i) \leq d, \quad i \in \{1,2\}.
$$
While $\cJ$ is coercive on $\cN$, the sequences are bounded and in view of the Sobolev embedding, they are bounded also in $L^2 (\R^N) \times L^2 (\R^N)$, say 
$$
|u_n^1|_2 + |v_n^1|_2 + |u_n^2|_2 + |v_n^2|_2 \leq M.
$$
We will consider two cases. \\
\textbf{Case 1.} $|u_n^1 - u_n^2|_p + |v_n^1 - v_n^2|_p \to 0$. \\
Fix any $\varepsilon > 0$ and note that
\begin{align*}
&\quad \left\| (u_n^1 - u_n^2, v_n^1 - v_n^2) \right\|^2 = \cJ'(u_n^1, v_n^1) (u_n^1-u_n^2, v_n^1-v_n^2) - \cJ'(u_n^2, v_n^2) (u_n^1-u_n^2, v_n^1-v_n^2) \\
&\quad + \int_{\R^N} (f_1 (u_n^1) - f_1 (u_n^2) ) (u_n^1 - u_n^2) \, dx + \int_{\R^N} (f_2 (v_n^1) - f_2 (v_n^2) ) (v_n^1 - v_n^2) \, dx \\
&\quad - \int_{\R^N} \left( |u_n^1 |^{q-2} u_n^1 - |u_n^2|^{q-2} u_n^2 \right) (u_n^1 - u_n^2) \, dx - \int_{\R^N} \left( |v_n^1 |^{q-2} v_n^1 - |v_n^2|^{q-2} v_n^2 \right) (v_n^1 - v_n^2) \, dx \\
&\quad + 2 \int_{\R^N} \lambda(x) (u_n^1 - u_n^2)(v_n^1 - v_n^2) \, dx \\
&\leq \varepsilon \| (u_n^1 - u_n^2, v_n^1 - v_n^2) \| \\
&\quad + \varepsilon \int_{\R^N} (|u_n^1| + |u_n^2|) |u_n^1 - u_n^2| \, dx + C_\varepsilon \int_{\R^N} (|u_n^1|^{p-1} + |u_n^2|^{p-2}) |u_n^1 - u_n^2| \, dx \\
&\quad + \varepsilon \int_{\R^N} (|v_n^1| + |v_n^2|) |v_n^1 - v_n^2| \, dx + C_\varepsilon \int_{\R^N} (|v_n^1|^{p-1} + |v_n^2|^{p-2}) |v_n^1 - v_n^2| \, dx \\
&\quad - \int_{\R^N} \left( |u_n^1 |^{q-2} u_n^1 - |u_n^2|^{q-2} u_n^2 \right) (u_n^1 - u_n^2) \, dx - \int_{\R^N} \left( |v_n^1 |^{q-2} v_n^1 - |v_n^2|^{q-2} v_n^2 \right) (v_n^1 - v_n^2) \, dx \\
&\quad + 2 \int_{\R^N} \lambda(x) (u_n^1 - u_n^2)(v_n^1 - v_n^2) \, dx \\
&\leq (1 + C_0) \varepsilon \| (u_n^1 - u_n^2, v_n^1 - v_n^2) \| + D_\varepsilon \left( |u_n^1 - u_n^2|_p + |v_n^1 - v_n^2|_p \right) \\
&\quad + C_1 \left( |u_n^1 - u_n^2|_q + |v_n^1 - v_n^2|_q \right) + 2 \int_{\R^N} \lambda(x) (u_n^1 - u_n^2)(v_n^1 - v_n^2) \, dx
\end{align*}
for $C_0, C_1, D_\varepsilon > 0$. From our assumption we have
$$
|u_n^1 - u_n^2|_p + |v_n^1 - v_n^2|_p \to 0.
$$
Since $(u_n^1 - u_n^2)$ and $(v_n^1 - v_n^2)$ are bounded in $L^2(\R^N)$ and $2 <q <p$, it follows from the interpolation inequality that there holds
$$
|u_n^1 - u_n^2|_q + |v_n^1 - v_n^2|_q \to 0.
$$
Taking \eqref{1.2} into account we get
$$
\| (u_n^1 - u_n^2, v_n^1 - v_n^2) \|^2 \leq \varepsilon \frac{(1+C_0)}{1-\delta} \| (u_n^1 - u_n^2, v_n^1 - v_n^2) \| + o(1)
$$
for all $\varepsilon > 0$. Hence
$$
\limsup_{n\to\infty} \| (u_n^1 - u_n^2, v_n^1 - v_n^2) \|^2 \leq \varepsilon \frac{(1+C_0)}{1-\delta} \limsup_{n\to\infty}\| (u_n^1 - u_n^2, v_n^1 - v_n^2) \|
$$
and therefore $\| (u_n^1 - u_n^2, v_n^1 - v_n^2) \| \to 0$. From Lemma \ref{lem:lip} we obtain
$$
\|(w_n^1 - w_n^2, z_n^1 - z_n^2)\| = \| m^{-1} (u_n^1, v_n^1) - m^{-1} (u_n^2, v_n^2) \| \leq L \| (u_n^1 - u_n^2, v_n^1 - v_n^2) \| \to 0.
$$
\textbf{Case 2.} $|u_n^1 - u_n^2|_p + |v_n^1 - v_n^2|_p \not\to 0$. \\
In view of Lions lemma (see Lemma \ref{lem:lions}) there is a sequence $(y_n) \subset \R^N$ such that
$$
\int_{B(y_n,1)} |u_n^1 - u_n^2|^2 \, dx + \int_{B(y_n,1)} |v_n^1 - v_n^2|^2 \, dx \geq \varepsilon
$$
for some $\varepsilon > 0$. In view of $\tau_k$-invariance of $\cN$, $\cJ$, $\cJ \circ m$ and $\tau_k$-equivariance of $\nabla \cJ$, $\nabla (\cJ \circ m)$, $m$ and $m^{-1}$ we can assume that the sequence $(y_n) \subset \R^N$ is bounded. We have that, up to a subsequence
$$
(u_n^i, v_n^i) \weakto (u^i, v^i) \quad \mathrm{in} \ E,  \ i \in \{1,2\} 
$$
and $(u_n^1, v_n^1) \neq (u_n^2, v_n^2)$. Moreover $\cJ'(u^1,v^1) = \cJ'(u^2,v^2) = 0$ and
$$
\| (u_n^i, v_n^i) \| \to \alpha^i, \quad i \in \{1,2\}.
$$
We see that $\alpha^i$ satisfies
$$
0 < \beta := \inf_{(u,v) \in \cN} \| (u,v) \| \leq \alpha^i \leq \nu(d) := \sup \{ \| (u,v) \| \ : \ (u,v) \in \cN, \ \cJ(u,v) \leq d \}.
$$
Suppose that $(u^1, v^1) \neq (0,0)$ and $(u^2, v^2) \neq (0,0)$. Then $(u^i, v^i) \in \cN$, $(w^i, z^i) := m^{-1} (u^i, v^i) \in S$ and $(w^1, z^1) \neq (w^2, z^2)$. Then
\begin{align*}
\liminf_{n\to\infty} \| (w_n^1-w_n^2, z_n^1-z_n^2) \| &= \liminf_{n\to\infty} \left\| \frac{(u_n^1,v_n^1)}{\|(u_n^1,v_n^1)\|} - \frac{(u_n^2,v_n^2)}{\|(u_n^2,v_n^2)\|} \right\| \geq \left\| \frac{(u^1,v^1)}{\alpha^1} - \frac{(u^2,v^2)}{\alpha^2} \right\| \\ &= \| \beta_1 (w^1, z^1) - \beta_2 (w^2, z^2) \|,
\end{align*}
where $ \beta_i = \frac{\|(u^i,v^i)\|}{\alpha^i} \geq \frac{\beta}{\nu (d)}$, $i\in \{1,2\}$. Moreover
$$
\| (w_n^1, z_n^1) \| = \| (w_n^2, z_n^2) \| = 1.
$$
Hence
$$
\liminf_{n\to\infty} \| (w_n^1-w_n^2, z_n^1-z_n^2) \| \geq \| \beta_1 (w^1, z^1) - \beta_2 (w^2, z^2) \| \geq \min_{i \in \{1,2\}} \{ \beta^i \}  \| (w^1, z^1) - (w^2, z^2) \| \geq \frac{\beta \kappa}{\nu (d)}.
$$
If $(u^2, v^2) = (0,0)$ we have $(u^1, v^1) \neq (u^2, v^2) = (0,0)$ and similarly
$$
\liminf_{n\to\infty} \| (w_n^1-w_n^2, z_n^1-z_n^2) \| = \liminf_{n\to\infty} \left\| \frac{(u_n^1,v_n^1)}{\|(u_n^1,v_n^1)\|} - \frac{(u_n^2,v_n^2)}{\|(u_n^2,v_n^2)\|} \right\| \geq \left\| \frac{(u^1,v^1)}{\alpha^1} \right\| \geq \frac{\beta}{\nu(d)}.
$$
\end{proof}

In view of \cite[Lemma II.3.9]{Struwe} $\cJ \circ m \rightarrow \R$ admits a pseudo-gradient vector field, i.e. there is a Lipschitz continuous function $\cH : S \setminus \mathscr{C} \rightarrow T S$ such that
\begin{align*}
\cH (w) &\in T_w S, \\
\| \cH (w) \| &< 2 \| \nabla (\cJ \circ m)(w) \|, \\
\langle \mathcal{H}(w), \nabla  (\cJ \circ m)(w) \rangle &> \frac{1}{2} \| \nabla (\cJ \circ m)(w) \|^2
\end{align*}
for $w \in S \setminus \mathscr{C}$. Then we can define the flow $\eta : \cT \rightarrow S \setminus \mathscr{C}$ by
$$
\left\{ \begin{array}{l}
\frac{d \eta}{dt}(t,w) = - \mathcal{H}(\eta(t,w)), \\
\eta (0,w) = w ,
\end{array} \right.
$$
where $\cT := \{ (t,w) \ : \ w \in S \setminus \mathscr{C}, \ T^{-}(w) < t < T^+ (w) \}$. $T^-(w)$ and $T^+(w)$ are the maximal existence time in negative and positive direction of $t \mapsto \eta (t,w)$. Then we can repeat the arguments from the proof of \cite[Theorem 1.2]{SzulkinWeth} and \cite[Theorem 1.2]{Bieganowski}. In fact we show that for any $k \geq 1$ there exists $(w_k,z_k) \in S$ such that
$$
(\cJ \circ m)' (w_k, z_k) = 0 \quad \mbox{and} \quad \cJ ( m (w_k, z_k) ) = c_k,
$$
where
$$
c_k := \inf \left\{ d \in \R \ : \ \gamma \left( \left\{ \left( w,z \right) \in S \ : \ \cJ \left( m \left( w,z \right) \right) \leq d \right\} \right) \geq k \right\}
$$
and $\gamma$ denotes the Krasnoselskii genus for closed and symmetric sets. We refer to \cite{Struwe} for basic facts about the Krasnoselskii genus and Lusternik-Schnirelmann values. Moreover $c_k < c_{k+1}$ and we have a contradiction with the assumption that $\mathcal{F}$ is finite.
\newline
\newline
\textbf{Acknowledgements.} The author was partially supported by the National Science Centre, Poland (Grant No. 2017/25/N/ST1/00531) and he would like to thank the referee for many valuable comments helping to improve the paper.



\begin{thebibliography}{99}
\baselineskip 2 mm



\bibitem{AmbrosettiCeramiRuiz} A. Ambrosetti, G. Cerami, D. Ruiz: {\em Solitons of linearly coupled systems of semilinear non-autonomous equations on $\mathbb{R}^N$}, J. Funct. Anal. \textbf{254}, (2008) 2816--2845.

\bibitem{AmbrosettiColorado} A. Ambrosetti, E. Colorado: {\em Bound and ground states of coupled nonlinear Schr\"odinger equations}, C. R. Math. Acad. Sci. Paris Ser. I \textbf{342}, (2006), 453--458.


\bibitem{BartschDingPeriodic} T. Bartsch, Y. Ding: {\em On a nonlinear Schr\"odinger equation with periodic potential}, Math. Ann. \textbf{313} (1999), no. 1, 15–37. 







\bibitem{Bartsch} T. Bartsch: {\em Infinitely many solutions of a symmetric Dirichlet problem}, Nonlinear Analysis, Vol. \textbf{20}, Issue 10 (1993), p. 1205--1216.

\bibitem{Bieganowski} B. Bieganowski: {\em Solutions of the fractional Schrödinger equation with a sign-changing nonlinearity}, J. Math. Anal. Appl. \textbf{450} (2017), 461--479.

\bibitem{BieganowskiMederski} B. Bieganowski, J. Mederski: {\em Nonlinear Schr\"odinger equations with sum of periodic and vanishig potentials and sign-changning nonlinearities}, Commun. Pure Appl. Anal., Vol. \textbf{17}, Issue 1, (2018), p. 143--161.




\bibitem{BenciGrisantiMeicheletti} V. Benci, C.R. Grisanti, A.M. Micheletti: {\em Existence and non existence of the ground state solution for the nonlinear Schr\"odinger equations with $V(\infty) = 0$}, Topol. Methods in Nonlinear Anal. \textbf{26}, (2005), 203--219.





\bibitem{BuffoniJeanStuart} B. Buffoni, L. Jeanjean, C. A. Stuart: {\em Existence of a nontrivial solution to a strongly indefinite semilinear equation}, Proc. Amer. Math. Soc. \textbf{119} (1993), no. 1, 179--186. 


\bibitem{ChabrowskiSzulkin2002} J. Chabrowski, A. Szulkin: {\em On a semilinear Schr\"odinger equation with critical Sobolev exponent}, Proc. Amer. Math. Soc. \textbf{130}, (2002), 85--93.

\bibitem{ChenZouCalPDE2013} Z. Chen, W. Zou: {\em An optimal constant for the existence of least energy solutions of a coupled Schr\"odinger system}, Calc. Var. PDE. \textbf{48}, (2013), No.3-4, 695--711.




\bibitem{CotiZelati} V. Coti-Zelati, P. Rabinowitz: {\em Homoclinic type solutions for a semilinear elliptic PDE on $\R^n$}, Comm. Pure Appl. Math. \textbf{45}, (1992), no. 10, 1217--1269.


\bibitem{doO} J.M. do \'{O}, J.C. de Albuquerque: {\em Coupled elliptic systems involving the square root of the Laplacian and Trudinger-Moser critical growth}, Differential Integral Equations, Volume 31, Number 5/6 (2018), 403--434.

\bibitem{Doerfler} W. D\"orfler, A. Lechleiter, M. Plum, G. Schneider, C. Wieners, {\em Photonic Crystals: Mathematical Analysis and Numerical Approximation}, Springer Basel (2012)


\bibitem{GoodmanWinsteinJNS2001} R.H. Goodman, M.1. Weinstein, P.J. Holmes: {\em Nonlinear Propagation of Light in One-Dimensional Periodic Structures}, J. Nonlinear Science \textbf{11}, (2001), 123--168.

\bibitem{GuoMederski} Q. Guo, J. Mederski: {\em Ground states of nonlinear Schr\"odinger equations with sum of periodic and inverse-square potentials}, Journal of Differential Equations \textbf{260}, (2016), 4180--4202

\bibitem{IkomaTanaka} N. Ikoma, K. Tanaka: {\em A local mountain pass type result for a system of nonlinear Schr\"odinger equations}, Calc. Var. Partial Differ. Equ. \textbf{40}, (2011), 449--480.



\bibitem{Kuchment} P. Kuchment: {\em The mathematics of photonic crystals}, Mathematical modeling in optical science, Frontiers Appl. Math., \textbf{22}, SIAM, Philadelphia (2001), 207--272.

\bibitem{KryszSzulkin} W. Kryszewski, A. Szulkin, {\em Generalized linking theorem with an application to semilinear Schr\"odinger equation}, Adv. Diff. Eq. \textbf{3}, (1998), 441--472.



\bibitem{LiSzulkin} G. Li, A. Szulkin: {\em An asymptotically periodic Schr\"odinger equation with indefinite linear part}, Commun. Contemp. Math. \textbf{4}, (2002), no. 4, 763--776.

\bibitem{LiTang} G. Li, X.H. Tang: {\em Nehari-type ground state solutions for Schr\"{o}dinger equations including critical exponent}, Appl. Math. Lett. \textbf{37} (2014), 101-106.







\bibitem{Liu} S. Liu: {\em On superlinear Schr\"odinger equations with periodic potential}, Calc. Var. Partial Differential Equations \textbf{45} (2012), no. 1-2, 1--9.

\bibitem{MaiaJDE2006} L.A. Maia, E. Montefusco, B. Pellacci: {\em Positive solutions for a weakly coupled nonlinear Schr\"odinger system}, J. Differential Equations \textbf{229} (2), (2006), 743--767.

\bibitem{Malomed} B. Malomed: {\em Multi-component Bose-Einstein condensates: Theory}, In: Emergent Nonlinear Phenomena in Bose-Einstein Condensation, P. G. Kevrekidis et al. (eds.), Atomic, Optical, and Plasma Physics 45, Springer-Verlag, Berlin, 2008, 287--305.


\bibitem{MederskiTMNA2014} J. Mederski: {\em Solutions to a nonlinear Schr\"odinger equation with periodic potential and zero on the boundary of the spectrum}, Topol. Methods Nonlinear Anal. \textbf{46} (2015), no. 2, 755--771.

\bibitem{MederskiNLS2014} J. Mederski: {\em Ground states of a system of nonlinear Schr\"odinger equations with periodic potentials},  Comm. Partial Differential Equations \textbf{41} (2016), no. 9, 1426--1440.



\bibitem{Pankov} A. Pankov: {\em Periodic Nonlinear Schr\"odinger Equation with Application to Photonic Crystals}, Milan J. Math. \textbf{73}, (2005), 259--287.

\bibitem{PankovDecay} A. Pankov: {\em On decay of solutions to nonlinear Schr\"odinger equations},  Proc. Amer. Math. Soc. \textbf{136}, (2008), 2565--2570.


\bibitem{PengChenTang} J. Peng, S. Chen, X. Tang: {\em Semiclassical solutions for linearly coupled Schr\"{o}dinger equations without compactness}, Complex Variables and Elliptic Equations (2018), DOI: 10.1080/17476933.2018.1450395


\bibitem{Rabinowitz:1992} P.H. Rabinowitz: {\em On a class of nonlinear Schr\"odinger equations}, Z. Angew. Math. Phys. \textbf{43}, (1992), 270--291.





\bibitem{NonlinearPhotonicCrystals} R. E. Slusher, B. J. Eggleton: {\em Nonlinear Photonic Crystals}, Springer 2003.

\bibitem{Struwe} M. Struwe: {\em Variational Methods}, Springer 2008.


\bibitem{SzulkinWeth} A. Szulkin, T. Weth: {\em Ground state solutions for some indefinite variational problems}, J. Funct. Anal. \textbf{257}, (2009), no. 12, 3802--3822. 






\bibitem{Willem} M. Willem: {\em Minimax Theorems}, Birkh\"auser Verlag 1996.

\bibitem{WillemZou} M. Willem, W. Zou: {\em On a Schr\"odinger equation with periodic potential and spectrum point zero}, Indiana Univ. Math. J. \textbf{52}, (2003), no. 1, 109--132.


\bibitem{Zhang} H. Zhang, J. Xu, F. Zhang: {\em Existence of positive ground states for some nonlinear Schr\"{o}dinger systems}, Bound. Value Probl., (2013), 1--16.

\end{thebibliography}
\end{document}